\documentclass[reqno]{amsart} 
\usepackage{amsfonts, amsmath, amsthm, amssymb, latexsym, graphicx}

\theoremstyle{plain}
\newtheorem{theorem}{Theorem}[section]
\newtheorem{corollary}[theorem]{Corollary}
\newtheorem{lemma}[theorem]{Lemma}
\newtheorem{proposition}[theorem]{Proposition}
\theoremstyle{definition}

\newtheorem{remark}[theorem]{Remark}

\theoremstyle{remark}

\numberwithin{equation}{section}

\title[]{Some lower bounds for distances of  roots of  classical orthogonal polynomials}

\author{Michael Voit}
\address{Fakult\"at Mathematik, Technische Universit\"at Dortmund,
          Vogelpothsweg 87,
          D-44221 Dortmund, Germany}
\email{michael.voit@math.tu-dortmund.de}


\begin{document}
\date{\today}

\begin{abstract}
  Let $(P_N)_{N\ge0}$ one of the classical sequences of orthogonal polynomials, i.e.,
 Hermite, Laguerre or Jacobi polynomials.
 For the roots $z_{1,N},\ldots, z_{N,N}$ of $P_N$ we derive lower estimates for 
 $\min_{i\ne j}|z_{i,N}-z_{j,N}|$ and the distances
  from the boundary of the orthogonality intervals. The proofs are based on recent
 results on the eigenvalues of
  the covariance matrices
  in central limit theorems for  associated $\beta$-random matrix ensembles
 where these  entities appear as entries, and
  where the eigenvalues of these matrices are known.
\end{abstract}

\maketitle

\section{Introduction}

Let $(P_N)_{N\ge0}$ one of the classical sequences of orthogonal polynomials, i.e., Hermite polynomials $(H_N)_{N\ge0}$, 
Laguerre polynomials   $L_N^{(\nu-1)}$ with $\nu>0$, or Jacobi polynomials
$(P_N^{(\alpha,\beta)})_{N\ge0}$ with $\alpha,\beta>-1$.
We  derive  lower estimates for the minimal distance $M_N:=\min_{1\le i<j\le N}(z_{i,N}-z_{j,N})$
for the ordered roots $z_{1,N}\ge\ldots\ge z_{N,N}$ of $P_N$ depending on $N$ and the other parameters $\nu,\alpha,\beta$.
The proof will be different from other estimates for $M_N$ in the literature
where often methods like Sturm's comparison theorem
are used; see e.g.~\cite{DJ, JT, K1, K2, H} and Ch. 6 of \cite{S}.

We here discuss the following approach:
For all three classes above, the roots of $P_N$ appear for  $\beta\to\infty$ in  limit formulas for  
classical $\beta$-random matrix ensembles (where this $\beta>0$ has no connection to the $\beta$ of the Jacobi polynomials)
 where $N$  is a fixed dimension parameter.
For the general background on  random matrix ensembles we refer to \cite{D, Me, F}. 
For  these classical ensembles there exist associated central limit theorems (CLTs) in the freezing regime
with explicit formulas for the covariance matrices $\Sigma_N$
  and their inverses $\Sigma_N^{-1}$
for the limit Gaussian distributions of dimension $N$; see \cite{ AHV, AV,  DE2, GK, V} in the Hermite and Laguerre case
and \cite{ AHV, HV} in the Jacobi case. In particular, by \cite{ AV, HV, V},   the entries of $\Sigma_N^{-1}$
are given  in terms of the distances $z_{i,N}-z_{j,N}$, and the eigenvalues of $\Sigma_N^{-1}$ have a  simple form.
This  allows to compute the traces $tr(\Sigma_N^{-k})$ of powers of $\Sigma_N^{-1}$. Some simple computations then
will lead to
the estimates for the minimal distances $M_N$ and for the distances of the roots from the
boundary of the orthogonality measure. In the Laguerre case, some connection between the roots of the polynomials
and the associated ensembles also appear in \cite{CD, K1, K2}.

We  mention that, depending on relations between the parameters and the order $N$ of the  Laguerre and Jacobi polynomials, our estimates
are better or worse than those in the literature mentioned above.
For the details see  Remarks  \ref{remark-lagu}, \ref{rem-lagu2}, \ref{rem-jac1}, and \ref{remark-jacobi-end} below. Furthermore,
in the Hermite case, our results are always worse than the estimate
 $$  z_{i,N}-z_{i+1,N}\ge 2/\sqrt N \quad\quad(i=1,\ldots,N)$$
which follows from \cite{K2}. We thus point out that the complicated  inequalities in Proposition \ref{distance-main-a},
Lemma \ref{distance-main-lemma-b}, and Lemma \ref{distance-main-lemma-jac} below should be seen as the main results of this paper
where our lower estimates for the minimal distances and the distances from the boundary are just corollaries from these
inequalities.

This paper is organized as follows:  Section 2  contains the Hermite case,
Section 3 is devoted to the  Laguerre case, and  the
Jacobi case is treated in Section 4.

We would like to thank the two referees for their very useful comments.

 \section{The Hermite case}

 We first recapitulate some facts on $\beta$-Hermite ensembles.

 Let  $\beta=2k>0$ be a constant (both parameters are used in the the literature;
see e.g. \cite{AV, D,  DE1, DE2, F, GK, Me, V}).
Define the associated  Hermite ensemble as a  random vector $X_{k,N}$
with values in the closed Weyl chamber
$C_N^A:=\{x\in \mathbb R^N: \quad x_1\ge x_2\ge\ldots\ge x_N\}$ with Lebesgue density
\begin{equation}\label{density-A}
c_k^A   e^{-\|x\|_2^2/2} \prod_{1\le i<j\le N}(x_i-x_j)^{2k}
\end{equation}
with the well-known normalization (see e.g.~the survey \cite{FW})
\begin{equation}\label{const-A}
 c_k^A:= \Bigl(\int_{C_N^A}  e^{-\|y\|_2^2/2}\cdot \prod_{i<j} (y_i-y_j)^{2k} \> dy\Bigr)^{-1}
 =\frac{N!}{(2\pi)^{N/2}} \cdot\prod_{j=1}^{N}\frac{\Gamma(1+k)}{\Gamma(1+jk)}.
\end{equation}
We now compare $X_{k,N}$ with $\sqrt{2k}\cdot {\bf z}_N\in C_N^A$, where
the entries of $$ {\bf z}_N=(z_{1,N},\ldots, z_{N,N})\in C_N^A$$ 
are
 the ordered zeros of the classical  Hermite polynomial $H_N$ 
where, as usual (see e.g.\cite{S}),   $(H_N)_{N\ge 0}$ is orthogonal w.r.t.
the density  $e^{-x^2}$.
The basis for our estimations for  the distances of the roots  will be the following CLT:

\begin{theorem}\label{clt-main-a}
For each $N\ge 2$, the random variables
$X_{k,N} -  \sqrt{2k}\cdot {\bf z}_N$
converge for $k\to\infty$ to the  $N$-dimensional centered  normal distribution $N(0,\Sigma_N)$
with the regular covariance matrix $\Sigma_N$ with $\Sigma_N^{-1}=S_N=(s_{i,j})_{i,j=1}^N$ and
\begin{equation}\label{covariance-matrix-A}
s_{i,j}:=\left\{ \begin{array}{r@{\quad\quad}l}  1+\sum_{l\ne i} (z_{i,N}-z_{l,N})^{-2} & \text{for}\quad i=j \\
   -(z_{i,N}-z_{j,N})^{-2} & \text{for}\quad i\ne j  \end{array}  \right.  . 
\end{equation}
The matrix $S_N$ has the eigenvalues $1,2,\ldots, N$.
\end{theorem}

\begin{proof}[Remarks on the proof] The CLT was first derived by  Dumitriu, Edelman  \cite{DE2}  by using their tridiagonal random matrix models in \cite{DE1}
  with explicit formulas for  $\Sigma_N$. It was then reproved with  different methods in \cite{V} with 
  Eq.~(\ref{covariance-matrix-A}) for  $\Sigma_N^{-1}$. The eigenvalues of $\Sigma_N^{-1}$  were given in \cite{AV}.
  We  remark that in \cite{AHV} the duality of finite  orthogonal polynomials in the sense of de Boor and Saff (see \cite{BS, I, VZ})
  was used 
  to compute the entries of  $\Sigma_N$  from   (\ref{covariance-matrix-A}).
  As a curiosity, the equations for the entries of  $\Sigma_N$ in \cite{AHV} look different from those in  \cite{DE2}.
  Numerical computation for small $N$ show that the entries in both cases are equal, but it seems to be difficult to verify these identities
   for all $N$ as in both representations complicated formulas depending on the  $z_{i,N}$ appear.
  Furthermore, \cite{GK} gives a third proof of the CLT with the  same formulas for   $\Sigma_N$
  as in \cite{AHV}.
\end{proof}

For our estimates for the distances of the roots we  need the following  fact:

\begin{lemma}\label{est-trace}
 Let $B=(b_{i,j})_{i,j=1,\ldots,N}\in \mathbb R^{N,N}$ be a  symmetric matrix. Then, for all integers $r\ge0$,
 \begin{equation}\label{inequality-trace1}
   tr(B^{2^r})\ge \sum_{i=1}^N b_{i,i}^{2^r}.
 \end{equation}
 
 Moreover, if in addition $B$ has the eigenvalue $0$ with the eigenvector $(1,\ldots,1)^T$, then, for all integers $r\ge0$,
 \begin{equation}\label{inequality-trace2}
   tr(B^{2^r})\ge \Bigl(\frac{N}{N-1}\Bigr)^{2^r-1}\sum_{i=1}^N b_{i,i}^{2^r}.
 \end{equation}
\end{lemma}

\begin{proof} Let $B^{2^r}=:(b_{i,j}^{(r)})_{i,j=1,\ldots,N}$. Then for  $r\ge0$ and $i=1,\ldots,N$,
  $$ b_{i,i}^{(r+1)} =\sum_{j=1}^N  (b_{i,j}^{(r)})^2\ge  (b_{i,i}^{(r)})^2.$$
  Hence, by induction,  $b_{i,i}^{(r)}\ge  b_{i,i}^{2^r}$. Summation over $i$ then yields (\ref{inequality-trace1}).

  Now assume in addition that $(1,\ldots,1)^T$ is an eigenvector of $B$ with eigenvalue $0$. Then this also holds
  for the matrices  $B^{2^r}$, i.e., for all $i$,
  $$b_{i,i}^{(r)}=-\sum_{j;j\ne i}b_{i,j}^{(r)}.$$
  A simple argument with the Cauchy-Schwarz inequality now shows that
  \begin{align}
    b_{i,i}^{(r+1)}   &=\sum_{i=1}^N  (b_{i,j}^{(r)})^2=  (b_{i,i}^{(r)})^2 +\sum_{j;j\ne i}(b_{i,j}^{(r)})^2 \notag\\
    &\ge \Bigl(1+\frac{1}{N-1}\Bigr) (b_{i,i}^{(r)})^2  =\frac{N}{N-1} (b_{i,i}^{(r)})^2.  \notag
    \end{align}
  This and induction imply that
  $$b_{i,i}^{(r)}\ge  \Bigl(\frac{N}{N-1}\Bigr)^{2^r-1} b_{i,i}^{2^r}.$$
 Summation  then yields (\ref{inequality-trace2}).
\end{proof}

With these ingredients we now derive the following result:

\begin{proposition}\label{distance-main-a}
  For all $N\ge2$ and $i=1,\ldots,N$,
  \begin{equation}\label{distance-main-a-ung}
\Bigl(\sum_{l; l\ne i} \frac{1}{(z_{i,N}-z_{l,N})^{2}}\Bigr)^2 + \sum_{l; l\ne i} \frac{1}{(z_{i,N}-z_{l,N})^{4}} \le 
    \frac{(N-1)^3}{N}.
    \end{equation}
\end{proposition}

\begin{proof} We write the matrix $S_N$ from  Theorem \ref{clt-main-a} as $S_N=I_N+A_N$ with the identity matrix $I_N$
  where, by  Theorem \ref{clt-main-a}, $A_N$ has the eigenvalues $0,1,\ldots, N-1$. Therefore, for all integers  $r\ge0$,
   $$tr(A_N^{2^r})= 0+1^{2^r}+2^{2^r}+\ldots+(N-1)^{2^r}\le(N-1)^{2^r+1}.$$
  On the other hand, if we write $A_N^2=(a_{i,j}^{(2)})_{i,j=1,\ldots,N}$, we obtain from (\ref{covariance-matrix-A}) that
  $$a_{i,i}^{(2)}=\Bigl(\sum_{l; l\ne i} \frac{1}{(z_{i,N}-z_{l,N})^{2}}\Bigr)^2 + \sum_{l; l\ne i} \frac{1}{(z_{i,N}-z_{l,N})^{4}}$$
for $i=1,\ldots,N$.
We now apply (\ref{inequality-trace2}) in Lemma \ref{est-trace} to the matrix $B:=A_N^2$ (notice that the assumptions are satisfied!)
and conclude that
\begin{equation}\label{est-tr-a}
  tr(A_N^{2^r})\ge \Bigl(\frac{N}{N-1}\Bigr)^{2^{r-1}-1} \sum_{i=1}^N(a_{i,i}^{(2)})^{2^{r-1}}.
\end{equation} 
Therefore, for all  integers  $r\ge0$,
$$ \frac{N-1}{N}\Bigl(\frac{N}{N-1}\Bigr)^{2^{r}}\sum_{i=1}^N\Bigl(a_{i,i}^{(2)}\Bigr)^{2^{r}} \le(N-1)\cdot ((N-1)^2)^{2^r}.$$
As the  condition $0\le x^{2^r}\le C$ for all $r\in\mathbb N$ with some constant $C>0$ implies that $x\le 1$, we conclude
from the terms with power $2^r$ that for all $i$,
\begin{equation}\label{final-ung-a}
 \frac{N}{N-1} a_{i,i}^{(2)}\le (N-1)^2.
\end{equation}
The equation for $a_{i,i}^{(2)}$ above now yields the first inequality. The second ineqality of the proposition  is trivial.
\end{proof}

Proposition \ref{distance-main-a} has the following  obvious consequences:

\begin{corollary} For all $i=1,\ldots,N$,
  \begin{equation}\label{final-ung-a2}
    \sum_{l: \> l\ne i} (z_{i,N}-z_{l,N})^{-4} \le \frac{(N-1)^3}{2N}\le \frac{(N-1)^2}{2}\end{equation}
  and
   \begin{equation}\label{final-ung-a3}
    \sum_{l: \> l\ne i} (z_{i,N}-z_{l,N})^{-2}\le\frac{(N-1)^{3/2}}{N^{1/2}}\le N-1. \end{equation} 
   In particular, for all $i<N$,
   \begin{equation}\label{final-ung-a4}
     z_{i,N}-z_{i+1,N}\ge \frac{(2N)^{1/4}}{(N-1)^{3/4}}\ge \frac{ 2^{1/4}}{(N-1)^{1/2}}.
\end{equation}
 \end{corollary}

\begin{remark}
  \begin{enumerate}
  \item[\rm{(1)}] It is well known that the order $O(N^{-1/2})$ in (\ref{final-ung-a4}), and thus in (\ref{final-ung-a2})
    and (\ref{distance-main-a-ung}), is sharp. However, the constant in (\ref{final-ung-a4}) is not optimal.
    For instance,
    it can be shown in an elementary way that $  z_{i,N}-z_{i+1,N}\ge 2/\sqrt N$ for all $i=1,\ldots,N$; 
    see \cite{K2}. On the other hand, for $N=2$, $z_{1,2}-z_{2,2}=\sqrt 2$ (see (5.5.4) in \cite{S}), i.e.,
    (\ref{final-ung-a4}) here is an equality.
  \item[\rm{(2)}] If we sum the left hand sides of  (\ref{distance-main-a-ung}) and  (\ref{final-ung-a3}) over $i$,
    we obtain
 \begin{equation}\label{final-g-a1}
   \sum_{i,l: \> i\ne l} (z_{i,N}-z_{l,N})^{-2} =N(N-1)/2\end{equation}
   and
   \begin{align}\label{final-g-a2}
     \sum_{i=1}^{N}&\Bigl(\sum_{l; l\ne i} \frac{1}{(z_{i,N}-z_{l,N})^{2}}\Bigr)^2 + \sum_{i,l; l\ne i} \frac{1}{(z_{i,N}-z_{l,N})^{4}} \notag\\
     &=
     0^2+1^2+\ldots+(N-1)^2 =
    \frac{N(N-1)(2N-1)}{6}.
   \end{align}
   In fact,  by Theorem \ref{clt-main-a}, both sides of (\ref{final-g-a1}) are equal to $tr(S_N-I_N)$, and
   both sides of (\ref{final-g-a2}) are equal to  $tr((S_N-I_N)^2)$.
   (\ref{final-g-a1}) and (\ref{final-g-a2}) show that  (\ref{final-ung-a3}) and  (\ref{distance-main-a-ung})
   are sharp up to a factor of size 2 and 3 respectively.
  \end{enumerate}
 \end{remark}

\section{The Laguerre case}

We  start with some facts on  $\beta$-Laguerre ensembles. 
Let $k_1,k_2>0$ be constants and $k=(k_1,k_2)$. Define
 the associated  Laguerre ensemble as random vector $X_{k,N}$ with values in the closed Weyl chamber
 $$C_N^B:=\{x\in \mathbb R^N: \quad x_1\ge x_2\ge\ldots\ge x_N\ge0\}$$
with the Lebesgue density
\begin{equation}\label{density-B}
c_k^B   e^{-\|x\|_2^2/2}\cdot \prod_{i<j}(x_i^2-x_j^2)^{2k_2}\cdot \prod_{i=1}^N x_i^{2k_1}\end{equation}
with the well-known normalization
\begin{align}\label{const-B}
  c_k^B:=&\Bigl(\int_{C_N^B}  e^{-\|y\|_2^2/2}\cdot \prod_{i<j}(x_i^2-x_j^2)^{2k_2}\cdot \prod_{i=1}^N x_i^{2k_1} \> dy\Bigr)^{-1}
 \\
=&\frac{N!}{2^{N(k_1+(N-1)k_2-1/2)}} \cdot\prod_{j=1}^{N}\frac{\Gamma(1+k_2)}{\Gamma(1+jk_2)\Gamma(\frac{1}{2}+k_1+(j-1)k_2)}.
\notag\end{align}
 We rewrite $k$ as
 $(k_1,k_2)=(\nu\cdot\beta,\beta)$ with $\nu,\beta>0$. 
For  $\nu>0$ fixed and  $\beta\to\infty$ we shall compare $X_{k,N}$ with the  vector
$\sqrt{\beta}\cdot {\bf r}_N\in C_N^B$ where ${\bf r}_N=(r_1,\ldots,r_N)$ satisfies
$(r_1^2,\ldots,r_N^2)=2{\bf z}_N$ where the entries of ${\bf z}_N=(z_{1,N},\ldots,z_{N,N})\in C_N^B$ are
the ordered zeros of the classical  Laguerre polynomial   $L_N^{(\nu-1)}$. Recapitulate that, as usual,
the  $L_N^{(\nu-1)}$ are orthogonal  w.r.t. the density $e^{-x}\cdot x^{\nu-1}$ on $]0,\infty[$  for $\nu>0$.
The following CLT will be crucial for  our estimates:

\begin{theorem}\label{clt-main-b}
Let $N\ge 1$ be an integer and $\nu>0$. Then 
$X_{k,N} -  \sqrt{\beta}\cdot {\bf r}_N$
converges for  $\beta\to\infty$ to the centered $N$-dimensional normal distribution $N(0,\Sigma_N)$
with the regular covariance matrix $\Sigma_N$ with  $\Sigma_N^{-1}=S_N=(s_{i,j})_{i,j=1,\ldots,N}$ with
\begin{equation}\label{covariance-matrix-B}
  s_{i,j}:=\left\{ \begin{array}{r@{\quad\quad}l}  1+ \frac{\nu}{z_{i,N}}+
    2\sum_{l; l\ne i}\frac{ z_{i,N}+ z_{l,N}}{ (z_{i,N}- z_{l,N})^2} &
 \text{for}\quad i=j \\
 -4\frac{\sqrt{z_{i,N}z_{j,N}} }{ (z_{i,N}- z_{j,N})^2} & \text{for}\quad i\ne j  \end{array}  \right.  . 
\end{equation}
The matrix $S_N$ has the eigenvalues $2,4,\ldots, 2N$.
\end{theorem}

\begin{proof}[Remarks on the proof] This CLT was  first  derived by  Dumitriu, Edelman  \cite{DE2}  via
  their tridiagonal random matrix models \cite{DE1}
  with explicit formulas for  $\Sigma_N$. The CLT was then reproved in a  different way in \cite{V} with 
  Eq.~(\ref{covariance-matrix-A}) for  $\Sigma_N^{-1}$ where  the $s_{i,j}$ there are expressed in terms of ${\bf r}_N$
  instead of  ${\bf z}_N$ by using
  \begin{align}\label{umrechnen-b}
 \frac{ z_{i,N}+ z_{l,N}}{ (z_{i,N}- z_{l,N})^2} &=  (r_i-r_l)^{-2}+ (r_i+r_l)^{-2} ,\notag\\
 -2\frac{\sqrt{z_{i,N}z_{j,N}} }{ (z_{i,N}- z_{j,N})^2} &= (r_i+r_j)^{-2}  -(r_i-r_j)^{-2}.
  \end{align}
 The eigenvalues  were determined in \cite{AV}, and
  in \cite{AHV} the duality of finite  orthogonal polynomials
  was used 
  to compute the entries of  $\Sigma_N$  from   (\ref{covariance-matrix-A}).
  Again, the equations for the entries of  $\Sigma_N$ in \cite{AHV} look different from those in  \cite{DE2}.
\end{proof}

We now use Theorem \ref{clt-main-b} as in Section 2 to derive estimates. This approach leads
to  preliminary results.

\begin{lemma}\label{distance-main-lemma-b}
  Let $\nu>0$. For all $N\ge1$, the ordered roots $z_{1,N}>\ldots > z_{N,N}>0$ of  $L_N^{(\nu-1)}$ satisfy
\begin{equation}\label{est-b01}\Bigl(\frac{\nu}{z_{i,N}}+
  2\sum_{l; l\ne i}\frac{ z_{i,N}+ z_{l,N}}{ (z_{i,N}- z_{l,N})^2}\Bigr)^2+ 16 \sum_{l; l\ne i}\frac{z_{i,N}z_{l,N} }{ (z_{i,N}- z_{l,N})^4}
  \le (2N-1)^2.
 \end{equation}
  In particular
 \begin{equation}\label{est-b02} z_{N,N}\ge \frac{\nu}{2N-1}.
 \end{equation}
\end{lemma}

\begin{proof} We proceed as in the proof of Theorem \ref{distance-main-a} and write 
 $S_N$  as $S_N=I_N+A_N$ 
  where, by  Theorem \ref{clt-main-b}, $A_N$ has the eigenvalues $1,3,5,\ldots, 2N-1$. Therefore, for all integers  $r\ge0$,
  \begin{equation}\label{est-b1}
    tr(A_N^{2^r})= 1^{2^r}+3^{2^r}+5^{2^r}+\ldots+(2N-1)^{2^r}\le N(2N-1)^{2^r}.
    \end{equation}
  On the other hand, if we write $A^2=:(a_{i,j}^{(2)})_{i,j=1,\ldots,N}$, we obtain from (\ref{covariance-matrix-B}) that
  $$a_{i,i}^{(2)}=\Bigl(\frac{\nu}{z_{i,N}}+
    2\sum_{l; l\ne i}\frac{ z_{i,N}+ z_{l,N}}{ (z_{i,N}- z_{l,N})^2}\Bigr)^2+ 16 \sum_{l; l\ne i}\frac{z_{i,N}z_{l,N} }{ (z_{i,N}- z_{l,N})^4}$$
    for $i=1,\ldots,N$.
    Moreover, (\ref{est-b1}) and  Lemma \ref{est-trace} for $B:=A_N^2$ lead to
\begin{equation}\label{est-b2}
 N((2N-1)^2)^{2^{r-1}} \ge tr(A_N^{2^r})\ge\sum_{i=1}^N(a_{i,i}^{(2)})^{2^{r-1}} \end{equation}
for all  integers  $r\ge1$, which implies that $a_{i,i}^{(2)}\le (2N-1)^2$ for all $i$. This proves (\ref{est-b01}) and also
(\ref{est-b02}).
\end{proof}

\begin{remark}\label{remark-sum}
Similar to the identities  (\ref{final-g-a1}) and  (\ref{final-g-a2}) in the Hermite case, we have in the Laguerre case
 \begin{equation}\label{final-g-b1}
 \nu\sum_{i=1}^N\frac{1}{z_{i,N}}+2 \sum_{i,l: \> i\ne l} \frac{z_{i,N}+z_{l,N}}{(z_{i,N}-z_{l,N})^{2}} =2+4+\ldots 2(N-1)=N(N-1)\end{equation}
   and
   \begin{align}\label{final-g-b2}
     \sum_{i=1}^{N}&
\Bigl(\Bigl(\frac{\nu}{z_{i,N}}+
    2\sum_{l; l\ne i}\frac{ z_{i,N}+ z_{l,N}}{ (z_{i,N}- z_{l,N})^2}\Bigr)^2+ 16 \sum_{l; l\ne i}\frac{z_{i,N}z_{l,N} }{ (z_{i,N}- z_{l,N})^4}\Bigr)
 \notag\\
     &=
    1^2+3^2+\ldots+(2N-1)^2 =
    \frac{N(2N-1)(2N+1)}{3}.
   \end{align}
   In fact,  both sides of (\ref{final-g-b1}) and (\ref{final-g-b2}) are equal to  $tr(S_N-I_N)$ and
   $tr((S_N-I_N)^2)$ respectively. (\ref{final-g-b2}) in particular shows that (\ref{est-b01}) is sharp up to a factor 3
   in the worst case.
  \end{remark}

\begin{remark}\label{remark-lagu} We briefly discuss  (\ref{est-b02}), which is a byproduct of   (\ref{est-b01}):
 \begin{enumerate}
\item[\rm{(1)}]  The order $O(N^{-1})$ in (\ref{est-b02}) is sharp by classical upper
  bounds for $z_{N,N}$ in \cite{H} or Eq.~(6.31.12) in \cite{S}.
\item[\rm{(2)}]  For $N=1$, we have equality in (\ref{est-b02}) by (5.1.6) in \cite{S}.
\item[\rm{(3)}] By  Eq.~(6.31.12) in \cite{S}, $z_{N,N}$ can be estimated in terms of the first zero of the Bessel function $J_{\nu-1}$.
  If one combines this estimate with Eq.~(5) in Section 15.3 of Watson \cite{W} on these  zeroes of $J_{\nu-1}$,
  one obtains 
\begin{equation}\label{est-b02-szegoe} z_{N,N}\ge \frac{\nu^2-1}{4(N+\nu/2)}.
 \end{equation}
For large $\nu$, (\ref{est-b02-szegoe}) is clearly better than  (\ref{est-b02}). However, 
for small $\nu>0$, (\ref{est-b02-szegoe}) is worse than (\ref{est-b02}).
The same holds for similar bounds in Theorem 1 of \cite{K3}.
\item[\rm{(4)}]  We finally notice that 18.16.12 in \cite{NIST} implies for $\nu\to\infty$ and $\nu>>N$ that
  $z_{N,N}\ge \nu +o(\nu)$.
\end{enumerate} \end{remark}

We next use Lemma \ref{distance-main-lemma-b} to derive estimates for  the distances  of  roots.

 \begin{proposition}\label{distance-main-theorem-b}
  Let $\nu>0$. For all $N\ge2$, the ordered roots $z_{1,N}>\ldots > z_{N,N}>0$ of  $L_N^{(\nu-1)}$ satisfy
  \begin{equation}\label{est-b03}
    z_{i,N}- z_{i+1,N}\ge\frac{\sqrt{2(1+\sqrt{1+8\nu^2})}}{2N-1}\ge \frac{2\cdot 2^{1/4}\sqrt \nu}{2N-1}
  \end{equation}
  and
  \begin{align}\label{est-b03a}  z_{i,N}- z_{i+1,N}\ge&
    \frac{\sqrt{2}}{2N-1}\sqrt{2+\sqrt{2}\cdot\sqrt{2+\frac{(2N-1)^2(\nu^2-1)^2}{(N+\nu/2)^2}}}\notag\\
    \ge&
 \frac{2^{3/4}\cdot \sqrt{\nu^2-1}}{\sqrt{(2N-1)(N+\nu/2)}}
  \end{align}
for $i=1,\dots,N-1$.
\end{proposition}

\begin{proof} 
  (\ref{est-b01}) and (\ref{est-b02}) imply
  \begin{align}\label{main-est-lagu-distances}
    (2N-1)^2\ge& 4\frac{ (z_{i,N}+ z_{i+1,N})^2}{ (z_{i,N}- z_{i+1,N})^4} + 16\frac{z_{i,N}z_{i+1,N} }{ (z_{i,N}- z_{i+1,N})^4} \\
  =& 4\frac{(z_{i,N}-z_{i+1,N})^2+8z_{i,N}z_{i+1,N} }{ (z_{i,N}- z_{i+1,N})^4} \ge
  4\frac{(z_{i,N}-z_{i+1,N})^2+\frac{8\nu^2}{(2N-1)^2} }{ (z_{i,N}- z_{i+1,N})^4}\notag
\end{align}
On the other hand, elementary calculus shows that for $w:= (z_{i,N}- z_{i+1,N})^2>0$ the inequality
$(2N-1)^2 w^2\ge 4w+\frac{32\nu^2}{(2N-1)^2}$ implies that
$$w\ge \frac{2(1+\sqrt{1+8\nu^2})}{(2N-1)^2}$$
which completes the proof of the first $\ge$ in (\ref{est-b03}) while the second one is trivial.
The proof of (\ref{est-b03a}) is analog by using (\ref{est-b02-szegoe}) in the last $\ge$ of (\ref{main-est-lagu-distances})
instead of 
 (\ref{est-b02}).
\end{proof}

\begin{remark}\label{rem-lagu2}
 \begin{enumerate}
\item[\rm{(1)}]  The order $O(N^{-1})$ in (\ref{est-b03}) and  (\ref{est-b03a}) is sharp for fixed $\nu$.
\item[\rm{(2)}] Clearly,  (\ref{est-b03a})  is better than  (\ref{est-b03}) for large $\nu$, while  (\ref{est-b03}) is the
  better for small $\nu$.
\item[\rm{(3)}]  We briefly compare   (\ref{est-b03a})
  with other  bounds in the literature which are of particular interest for $\nu>>N$.
  Theorem 3.1 of \cite{CD} implies
  \begin{equation}\label{lagu-alter1} z_{i,N}- z_{i+1,N}\ge \frac{\nu-1}{\sqrt{(N+\nu-1)N}},
    \end{equation}
  and results in \cite{K2} lead to
\begin{equation}\label{lagu-alter2}  z_{i,N}- z_{i+1,N}\ge \frac{2\sqrt{2}\cdot\nu}{\sqrt{(N+\nu)N}}.
    \end{equation}
(\ref{est-b03a}),   (\ref{lagu-alter1}), and  (\ref{lagu-alter2}) have the same order $\sqrt \nu/\sqrt N$ for $\nu>>N>>1$,
where the asymptotic constant in (\ref{lagu-alter2}) is the best one.
\item[\rm{(4)}] On the other hand, 
    Theorem 3.1 of \cite{JT} contains the bound
    $$ z_{i,N}- z_{i+1,N}\ge \frac{\pi\sqrt 2}{\sqrt{2\nu N +\nu+2N^2}}$$
    which is better than (\ref{est-b03}) and  (\ref{est-b03a}) for small $\nu$, but worse for large $\nu$.
     \item[\rm{(5)}] Notice that our  estimates for $ z_{i,N}- z_{i+1,N}$ depend heavily
       on good estimates for  $ z_{i,N},z_{i+1,N}$.
       In this way, one may use Theorem 6.31.3 in \cite{S} (see also 18.16.10 in \cite{NIST}) in combination with 
       estimates on the zeroes of Bessel functions in Section 15.3 in \cite{W}
       to derive better bounds from  (\ref{est-b01}) than in  (\ref{est-b03}) which then depend on $i$.
  \end{enumerate}
\end{remark}

The dependence of  Proposition \ref{distance-main-theorem-b} on good estimates for   $ z_{i,N},z_{i+1,N}$
is the motivation to compare the distances of consecutive roots   on a different scale by studying
the vectors  ${\bf r}_N$ above. In this way,  Theorem \ref{clt-main-b}
leads to the following estimate which is independent from $\nu$:

\begin{theorem}\label{distance-main-theorem-b-wurzel}
  Let $\nu>0$. For all $N\ge2$, the ordered roots $z_{1,N}>\ldots > z_{N,N}>0$ of  $L_N^{(\nu-1)}$ satisfy
  \begin{equation}\label{est-b04}  \sqrt{z_{i,N}}- \sqrt{ z_{i+1,N}}\ge\frac{1}{\sqrt{2N-1}}
 \quad (i=1,\dots,N-1).
 \end{equation}
\end{theorem}

\begin{proof} We use the numbers $r_i=\sqrt{2z_{i,N}}$ from the beginning of this section:
  By (\ref{umrechnen-b}), the matrix $S_N=(s_{i,j})_{i,j=1,\ldots,N}$ from Theorem \ref{clt-main-b} then satisfies
  $$s_{i,i}=1+\frac{2\nu}{r_i^2}+2\sum_{l: l\ne i}\Bigl( \frac{1}{(r_i-r_l)^2}+ \frac{1}{(r_i+r_l)^2}\Bigr)$$
  for all $i$; see also \cite{V}. If we apply the first statement of Lemma \ref{est-trace} to $B:=S_N-I_N$ with
  the eigenvalues $1,3,\ldots,2N-1$,we obtain from the methods of the proof of Theorem \ref{distance-main-a} that for all $i\ne l$,
  $$\frac{2}{(r_i-r_l)^2}\le 2N-1$$
  and thus the claim.
\end{proof}
 
\begin{remark} The Hermite and Laguerre polynomials  are related by
  $$H_{2N}(x)= c_N L_N^{(-1/2)}(x^2), \quad H_{2N+1}(x)= d_N x\cdot L_N^{(1/2)}(x^2)$$
  with some constants $c_N,d_N$ (see (5.6.1) of \cite{S}). If we compare Theorem \ref{distance-main-theorem-b-wurzel}
  (for $\nu=1/2,3/2$) with  (\ref{final-ung-a4}), we see that  (\ref{final-ung-a4}) is better by the factor $2^{1/4}$.
We are not able to explain this discrepancy precisely, but there should exist some connection, as
   the covariance matrices
  (written in terms of the $r_i$ as in \cite{V} in the Laguerre case)
 and their eigenvalues
  in these two cases admit
  obvious similarities.
\end{remark}

\section{The Jacobi case}

 We first recapitulate some facts on $\beta$-Jacobi ensembles; see e.g. \cite{F,  KN, Me, HV}.
 It  turns out that  it is convenient here to study these ensembles in a trigonometric form
 as in \cite{HV}.
 For this we fix $k=(k_1,k_2,k_3)\in[0,\infty[^3$ and consider  a  random vector $X_{k,N}$
 on the trigonometric alcove
$$ A_N:=\{t\in\mathbb R^N|\frac{\pi}{2}\geq t_1\geq...\geq t_N\geq 0 \}$$
with the Lebesgue density
\begin{equation}\label{density-joint-trig}
\tilde c_k\cdot \prod_{1\leq i< j \leq N}\left(\cos(2t_j)-\cos(2t_i)\right)^{k_3}
\prod_{i=1}^N\Bigl(\sin^{k_1}t_i\sin^{k_2}(2t_i)\Bigr)
\end{equation}
with a suitable Selberg normalization $\tilde c_k>0$ for parameters
$k=(k_1,k_2,k_3)\in[0,\infty[^3$; see  \cite{FW} for explicit formulas for $\tilde c_k>0$.
Following \cite{HV}, we  write
$$(k_1,k_2,k_3)=\kappa\cdot(a,b,1),$$
where $a\ge0$, $b>0$ are fixed, and where $\kappa$ tends to infinity. By \cite{HV}, the limit of the $X_{k,N}$ for
$\kappa\to\infty$ 
 can be described via the ordered zeros
of the  Jacobi polynomials 
$P_N^{(\alpha,\beta)}$ with 
\begin{equation}\label{jac-par}
  \alpha:=a+b-1>-1, \quad\beta=b-1>-1.\end{equation}
Please notice that the $\beta$ in (\ref{jac-par})
is  different from the  $\beta$ in  the preceding sections, and that the $\beta$ there is now called $\kappa$.
    We denote the ordered zeros of   $P_N^{(\alpha,\beta)}$ 
   by  $-1<z_{1,N}\le \ldots\le z_{N,N}<1$ and form the vectors 
   $$ {\bf z}_N:=(z_{1,N}, \ldots, z_{N,N})\quad\text{ and}\quad
 {\bf \tilde z}_N:= (\frac{1}{2}\arccos z_{1,N}, \ldots,\frac{1}{2}\arccos z_{N,N}). $$
The following CLT is shown in \cite{HV}:

\begin{theorem}\label{theoremCLT-transformed}
Let $X_{k,N}$ be $ A_N$-valued  random variables as described above.
 Then, for $\kappa\rightarrow\infty$,
 $$\sqrt{\kappa}( X_{N,k}-  {\bf\tilde z}_N ) $$
converges in distribution to  $N(0,\Sigma_N)$ where the  inverse 
of the  covariance matrix $\Sigma_N$ is given by
 $\Sigma_N^{-1}=: S_N=( s_{i,j})_{i,j=1,...,N}$ with
\begin{equation}\label{entry-jac}
 s_{i,j}=
\begin{cases}4\sum_{ l\ne j}
\frac{1-z_{j,N}^2}{(z_{j,N}-z_{l,N})^2}+2(a+b)\frac{1+z_{j,N}}{1-z_{j,N}}+2b\frac{1-z_{j,N}}{1+z_{j,N}}
&\textit{ for }i=j\\
\frac{-4\sqrt{(1-z_{j,N}^2)(1-z_{i,N}^2)}}{(z_{i,N}-z_{j,N})^2}&\textit{ for }i\neq j
\end{cases}.
\end{equation}
The matrix $ S_N$  has the  eigenvalues
$\lambda_j=2j(2N+\alpha+\beta+1-j)>0$ ($j=1,\ldots,N$).
\end{theorem}

In order to apply the methods of Sections 2 and 3 in the Jacobi case, we need the maximum
of the eigenvalues $\lambda_j$, i.e the spectral radius. For this  define 
\begin{equation} M:= M(\alpha,\beta,N):= \max_{j=1,\ldots,N}2j(2N+\alpha+\beta+1-j).
\end{equation}
Elementary calculus yields the following facts:

\begin{lemma} For all $\alpha,\beta>-1$,
\begin{equation}\label{max-ev1} M\le 2\Bigl(N+\frac{\alpha+\beta+1}{2}\Bigr)^2.
\end{equation}
Moreover, for $\alpha+\beta+1\ge0$,
\begin{equation}\label{max-ev2} M=\lambda_N=2N(N+\alpha+\beta+1).
\end{equation}
Furthermore, for all   $\alpha,\beta>-1$, $M(\alpha,\beta,1)=2(\alpha+\beta+2)$.
\end{lemma}

The methods of Section 3 lead to the following results:

\begin{lemma}\label{distance-main-lemma-jac}
  Let $\alpha,\beta>-1$. For  $N\ge1$ and $i=1,\dots,N$, the ordered  roots $1>z_{1,N}>\ldots > z_{N,N}>-1$ of  $P_N^{(\alpha,\beta)}$ satisfy
  \begin{align}\label{est-jac01}
    \Bigl(4&\sum_{l; l\ne i}\frac{1-z_{i,N}^2}{(z_{i,N}-z_{l,N})^2}+2(\alpha+1)\frac{1+z_{i,N}}{1-z_{i,N}}+2(\beta+1)\frac{1-z_{i,N}}{1+z_{i,N}}\Bigr)^2
   \notag\\ &+16\sum_{l; l\ne i}\frac{(1-z_{l,N}^2)(1-z_{i,N}^2)}{(z_{i,N}-z_{l,N})^4}
    \quad\le  M^2.
 \end{align}
\end{lemma}

\begin{proof}
By the definition of $M$, we have for all  integers  $r\ge0$,
\begin{equation}\label{est-j1}
    tr(S_N^{2^r})\le N \cdot M^{2^r}.
\end{equation}
On the other hand, if we write $S_N^2=:(s_{i,j}^{(2)})_{i,j=1,\ldots,N}$, we obtain from Theorem \ref{theoremCLT-transformed}
that for $i=1,\ldots,N$, $s_{i,i}^{(2)}$ is equal to the left hand side of (\ref{est-jac01}).
    Moreover, (\ref{est-j1}) and  Lemma \ref{est-trace} for $B:=S_N^2$ lead to
\begin{equation}\label{est-j2}
 N\cdot (M^2)^{2^{r-1}} \ge tr(S_N^{2^r})\ge\sum_{i=1}^N(s_{i,i}^{(2)})^{2^{r-1}} \end{equation}
for all  integers  $r\ge1$, which implies that $s_{i,i}^{(2)}\le M^2$ for all $i$ as claimed. 
\end{proof}

Lemma \ref{distance-main-lemma-jac} has the following consequences:

\begin{corollary}\label{distance-cor-jac1}
  For $\alpha,\beta>-1$ and $N\ge1$,
  \begin{equation}\label{est-j2-cor1}
    1-z_{N,N}\ge\frac{ 8(\alpha+1)}{ M + 4(\alpha+1)+ \sqrt{M^2-16(\alpha+1)(\beta+1)}}
    \ge \frac{ 4(\alpha+1)}{M+2(\alpha+1)} \end{equation}
  and
  \begin{equation}\label{est-j2-cor2}
    1+z_{1,N}\ge\frac{ 8(\beta+1)}{ M + 4(\beta+1)+ \sqrt{M^2-16(\alpha+1)(\beta+1)}}
\ge \frac{ 4(\beta+1)}{M+2(\beta+1)}
    .\end{equation}
\end{corollary}

\begin{proof} Lemma \ref{distance-main-lemma-jac} implies that
  \begin{equation}\label{help-jac1}
    2(\alpha+1)\frac{1+z_{i,N}}{1-z_{i,N}}+2(\beta+1)\frac{1-z_{i,N}}{1+z_{i,N}}\le M.
    \end{equation}
  Thus $x:=1-z_{1,N}\in ]0,2[$ satisfies
$$2(\alpha+1)\frac{2-x}{x}+2(\beta+1)\frac{x}{2-x}\le M$$
    which is, by elementary calculus, equivalent to
    $$(2(\alpha+\beta+2)+M)x^2 -2(4(\alpha+1)+M)x + 8(\alpha+1)\le0.$$
    This yields that
    $$x\in [x_-,x_+] \quad\text{for}\quad x_{\pm}=
    \frac{M + 4(\alpha+1)\pm \sqrt{M^2-16(\alpha+1)(\beta+1)}}{M+2(\alpha+\beta+2)}.$$
    As
    \begin{align}
x_-&= \frac{(M + 4(\alpha+1))^2- M^2+16(\alpha+1)(\beta+1)}{\Bigl(M+2(\alpha+\beta+2)\Bigr)
      \Bigl(M + 4(\alpha+1)+ \sqrt{M^2-16(\alpha+1)(\beta+1)} \Bigr) }\notag\\
    &=\frac{ 8(\alpha+1)}{
  M + 4(\alpha+1)+ \sqrt{M^2-16(\alpha+1)(\beta+1)}  },\notag
\end{align}
    the first $\ge$ in (\ref{est-j2-cor1}) follows. The second t $\ge$ there is obvious.
    Moreover,  (\ref{est-j2-cor2}) follows in the same way.
\end{proof}

\begin{remark}\label{rem-jac1}
\begin{enumerate}
\item[\rm{(1)}] The first $\ge$  in (\ref{est-j2-cor1}) and  (\ref{est-j2-cor2})
  are equalities for $N=1$ by the explicit form of  $P_N^{(\alpha,\beta)}$ in (4.21.2) of \cite{S}.
\item[\rm{(2)}] We briefly compare  (\ref{est-j2-cor1}) and  (\ref{est-j2-cor2}) with other known estimates.
  We first notice that  Theorem 6.3.2 of \cite{S} leads to an estimate for  $-1/2\le \alpha,\beta\le 1/2$, which in most cases
  is better than
  (\ref{est-j2-cor1}) and  (\ref{est-j2-cor2}) in this restricted case.

  Moreover, the asymptotic result 18.16.8 in \cite{NIST} for fixed $\alpha,\beta>-1/2$ and $N\to\infty$
  in combination with the estimate Eq.~(5) in Section 15.3 of Watson \cite{W} on the first zeros  of the Bessel functions
  $J_{\alpha}$ imply that
   \begin{equation}\label{est-j2-asympt}
    1-z_{N,N}\ge  \frac{\alpha(\alpha+2)}{2(N+(\alpha+\beta+1)/2)^2}+o(1/N^2). \end{equation}
   If we compare  this with (\ref{est-j2-cor1}), we obtain that, for fixed $\alpha,\beta$ and large $N$,
   (\ref{est-j2-asympt}) is better than  (\ref{est-j2-cor1})  for $\alpha> 1+\sqrt 5$, while the converse holds for 
   $-1/2<\alpha< 1+\sqrt 5$.
\end{enumerate}
\end{remark}

The preceding estimates also have the following variant:

\begin{corollary}\label{distance-cor-jac3} For all $i=1,\ldots,N$,
  \begin{equation}\label{est-j4}
    1-z_{i,N}^2\ge 2 \frac{ \min(\alpha+1,\beta+1)}{M}.
  \end{equation}
  Moreover, for $\alpha=\beta>-1$, 
  \begin{equation}\label{est-j5}
    1-z_{i,N}^2\ge \frac{ 8(\alpha+1)}{M+4(\alpha+1)}. \end{equation}
\end{corollary}

\begin{proof} Eq.~(\ref{help-jac1}) implies 
  $$  M(1-z_{i,N}^2)\ge  2(\alpha+1)(1+z_{i,N})^2+2(\beta+1)(1-z_{i,N})^2\ge2\cdot \min(\alpha+1,\beta+1)$$
  and thus (\ref{est-j4}).  Moreover, for $\alpha=\beta>-1$, Eq.~(\ref{help-jac1}) leads to
  $$4(\alpha+1)(1+z_{i,N}^2)\le M(1-z_{i,N}^2)$$
and thus to (\ref{est-j5}).
\end{proof}
  
We next estimate the distances of consecutive roots.

 \begin{theorem}\label{distance-main-theorem-jac}
  Let  $\alpha,\beta>-1$. For all $N\ge2$, the ordered roots $-1<z_{1,N}<\ldots < z_{N,N}<1$ of  $P_N^{(\alpha,\beta)}$   satisfy
  \begin{equation}\label{est-j03}  z_{i+1,N}- z_{i,N}\ge\frac{2^{7/4}}{M}\cdot (\min(\alpha+1,\beta+1))^{1/2}
 \quad (i=1,\dots,N-1).
  \end{equation}
  Moreover, for  $\alpha=\beta>-1$,
\begin{equation}\label{est-j04}  z_{i+1,N}- z_{i,N}\ge\frac{2^{11/4}(\alpha+1)^{1/2}}{\sqrt{M(M+4(\alpha+1))}}
 \quad (i=1,\dots,N-1).
  \end{equation}
  \end{theorem}

 \begin{proof} Lemma \ref{distance-main-lemma-jac} yields that
   $$ 16\sum_{l; l\ne i}
   \frac{(1-z_{i,N}^2)^2+(1-z_{l,N}^2)(1-z_{i,N}^2) }{(z_{i,N}-z_{l,N})^4}
    \le  M^2.$$
    This and (\ref{est-j4}) lead to
    $$2^7 \cdot \frac{(\min(\alpha+1,\beta+1))^2}{M^2(z_{i,N}-z_{i+1,N})^4} \le  M^2$$
and thus to (\ref{est-j03}). In the same way, (\ref{est-j5}) leads to (\ref{est-j04}).
 \end{proof}

\begin{remark}\label{remark-jacobi-end}
\begin{enumerate}
\item[\rm{(1)}]  If one compares Theorem \ref{distance-main-theorem-jac} with 
     Theorem 6.3.1 of \cite{S} for $-1/2\le \alpha,\beta\le 1/2$,
    then in most cases the estimate in  \cite{S} is again the better one under this restricton.
  \item[\rm{(2)}] If one uses the asymptotic result (\ref{est-j2-asympt}) for $\alpha,\beta$ fixed and $N$ large in the proof
    of the preceding theorem instead of (\ref{est-j4}), one gets an asymptotic modifications of Theorem \ref{distance-main-theorem-jac}
    where the asymptotic rate is slightly better than in (\ref{est-j03}) for large  $\alpha,\beta$.
\item[\rm{(2)}] In summary, we have the impression that our approach here
     on Jacobi ensembles and polynomials should be rewritten in some trigonometric form
    similar to the square-root-form in  Theorem  \ref{distance-main-theorem-b-wurzel} in the Laguerre case.
 Unfortunately, we were not able to transform
  the entries (\ref{entry-jac}) in the inverse covariance matrices $S_N$ above in a trigonometric and useful way.
\end{enumerate}
\end{remark}

\end{document}